\documentclass{amsart}

\usepackage{color, mathtools, verbatim, esint}

\mathtoolsset{showonlyrefs}

\newcommand\blfootnote[1]{%
  \begingroup
  \renewcommand\thefootnote{}\footnote{#1}%
  \addtocounter{footnote}{-1}%
  \endgroup
}

\newcommand{\ddb}{\sqrt{-1}\partial\overline{\partial}}
\newcommand{\dbar}{\overline{\partial}}

\renewcommand{\[}{\begin{equation} \begin{aligned} }
      \renewcommand{\]}{\end{aligned} \end{equation}}
\renewcommand{\phi}{\varphi}

\newtheorem{thm}{Theorem}
\newtheorem{prop}[thm]{Proposition}
\newtheorem{lem}[thm]{Lemma}
\newtheorem{cor}[thm]{Corollary}

\theoremstyle{definition}

\newtheorem{rem}[thm]{Remark}

\numberwithin{equation}{section}
\author{Gang Liu, G\'abor Sz\'ekelyhidi}
\address{Department of Mathematics, Northwestern University, Evanston, IL 60208}
\email{gangliu@math.northwestern.edu}
\address{Department of Mathematics, University of Notre Dame, Notre Dame, IN 46556}
\email{gszekely@nd.edu}
\title[Gromov-Hausdorff limits]{Gromov-Hausdorff limits of K\"ahler
  manifolds with Ricci curvature bounded below, II}
\date{}
\begin{document}
\begin{abstract}
We study non-collapsed Gromov-Hausdorff limits of K\"ahler manifolds
with Ricci curvature bounded below. Our main result is that each
tangent cone is homeomorphic to a normal affine variety. This extends
a result of Donaldson-Sun, who considered non-collapsed limits of polarized
K\"ahler manifolds with two-sided Ricci curvature bounds. 
\end{abstract}
\maketitle
\section{Introduction}
\blfootnote{The first author is partially supported by NSF grant DMS-1709894 and the Alfred P. Sloan Foundation. The second author is supported in part by NSF grant DMS-1350696}
Consider a sequence $(M_i, \omega_i, p_i)$ of pointed complete K\"ahler
manifolds of dimension $n$, with $\mathrm{Ric}(\omega_i) > -\omega_i$, and
$\mathrm{Vol}(B(p_i,1)) > \nu > 0$. Suppose that the sequence
converges to a metric space $(Z, d, p)$ in the pointed
Gromov-Hausdorff sense. By the work of
Cheeger-Colding~\cite{CC2}, and more recently
Cheeger-Jiang-Naber~\cite{CJN} and 
others, we have a detailed understanding of the structure of $Z$, even
if the $M_i$ are merely Riemannian. A starting point for this structure
theory is Cheeger-Colding's result~\cite{CC} that the limit space $Z$
admits tangent cones at each point, that are
metric cones. In this paper we are interested in studying the
additional structure of the tangent cones of $Z$ in the K\"ahler case. 

There are few general
results that exploit the K\"ahler condition: by
Cheeger-Colding-Tian~\cite{CCT}, the tangent cones are of the form
$C(Y)\times \mathbf{C}^k$, where $C(Y)$ does not split off a factor of
$\mathbf{R}$, while the first author~\cite{Liu1} showed that each
tangent cone admits a one-parameter group of isometries. Under
two-sided Ricci curvature bounds it follows from Anderson~\cite{And}
that the regular set in $Z$ is a K\"ahler manifold. In our
previous work~\cite{LiuSz} we showed that under Ricci lower bounds,
 for sufficiently small
$\epsilon > 0$ the $\epsilon$-regular set $\mathcal{R}_\epsilon\subset
Z$ has the structure of a complex manifold, although the metric may not be
smooth.  Here $x\in \mathcal{R}_\epsilon$ if 
$V_{2n} - \lim_{r\to 0} r^{-2n} \mathrm{vol}(B(x,r)) <
\epsilon$, where $V_{2n}$ is the volume of the Euclidean unit
ball. An important problem is whether the complex structure can be
extended across the singular set of $Z$ to give it the structure of an
analytic space. When the $(M_i,\omega_i)$ are polarized,
i.e. the $\omega_i$ are curvature forms of line bundles, then
Donaldson-Sun~\cite{DS1} (under two sided Ricci bounds), and the
authors~\cite{LiuSz} (under just lower Ricci bounds), showed that this
is the case. Without polarizations the question is still open, however
we have the following. 

\begin{thm}\label{thm:main}
  Every tangent cone of $Z$ is homeomorphic to a normal
  affine algebraic variety, such that under a suitable embedding into
  $\mathbf{C}^N$, the homothetic action on the tangent cone extends to
  a linear torus action.
\end{thm}

Theorem~\ref{thm:main} was shown previously by Donaldson-Sun~\cite{DS2} under the
assumptions that the $\omega_i$ are curvature forms of line bundles
$L_i\to M_i$, and $|\mathrm{Ric}(\omega_i)| < 1$. An important
application of their result is that in their setting the holomorphic
spectrum of the tangent cones is rigid, which in turn they used to show the
uniqueness of tangent cones. While we are not able to show uniqueness,
our result does imply the rigidity of the holomorphic spectrum under
two sided Ricci curvature bounds, even when the $(M_i,\omega_i)$ are
not polarized.

More precisely, recall that for a K\"ahler cone $C(Y)$,
possibly with singularities, the holomorphic spectrum is defined by
\[ \mathcal{S} = \{\mathrm{deg}(f)\,:\, f \text{ is
    homogeneous and holomorphic on } C(Y)\} \subset
  \mathbf{R}. \]
We then have
\begin{cor}\label{cor:1}
   Suppose that we have two sided bounds $|\mathrm{Ric}(\omega_i)|<1$
   along the sequence above. Then for any $q\in Z$ the holomorphic
   spectrum of every tangent cone at $q$ is the same. In addition
   the volume ratio $V_{2n}^{-1}\mathrm{Vol}(B(o,1))$
  is an algebraic number for every tangent cone $(Z_q,o)$. 
\end{cor}
As in \cite{DS2}, the rigidity of the holomorphic spectrum follows
from the fact that the space of tangent cones at each point is
connected, and the holomorphic spectrum consists of algebraic
numbers. Note that these results hold in particular for tangent cones at
infinity of Calabi-Yau manifolds with Euclidean volume growth. 
 
The method of proof follows the overall strategy of
Donaldson-Sun~\cite{DS1,DS2} for constructing holomorphic functions on
limit spaces. A crucial difference is that in our case the
holomorphic functions on a tangent cone $Z_q$ are not obtained as
limits of holomorphic functions on smooth manifolds. Instead we prove a
version of H\"ormander's $L^2$-estimate on the tangent
cone, see Proposition~\ref{prop:L2existence}. 
In the setting of two-sided Ricci curvature bounds there are
substantial simplifications using that the singular set has
codimension four, and the $L^2$-estimate can be proven directly on the
tangent cone. We give this argument separately in
Section~\ref{sec:Ricflat}, since it may be of independent
interest. The case of lower Ricci bounds is treated in
Section~\ref{sec:L2est} by proving approximate versions of the
basic estimate on smooth spaces converging to the tangent cone. Using
this, in Section~\ref{sec:affine} we follow the strategy of
Donaldson-Sun~\cite{DS1,DS2} to prove Theorem~\ref{thm:main}. 

\subsection*{Acknowledgements} We thank Aaron Naber for useful
comments on the paper, and Mihai P\u{a}un for pointing out an error. 

\section{Two-sided Ricci bounds}\label{sec:Ricflat}

In this section, we show that when our sequence $(M_i,\omega_i)$ has
two-sided Ricci curvature bounds $|\mathrm{Ric}(\omega_i)|<1$, then we
can directly prove the H\"ormander $L^2$ estimate on the tangent
cone. While the result is subsumed by the more general setting treated
in Section~\ref{sec:L2est}, the proof may be of independent interest. 

Let $X$ be a tangent cone of the non-collapsed Gromov-Hausdorff limit of a 
sequence of complete K\"ahler manifolds with two-sided Ricci
curvature bounds. Let $S\subset X$ denote the singular set (see
Cheeger-Colding~\cite{CC2} for details on the structure of $X$). The
regular set $\mathcal{R} = X \setminus S$ admits a Ricci flat K\"ahler metric
$\omega = \frac{1}{2} \ddb r^2$, where $r$ is the distance from the
vertex of the cone $X$. Suppose that we have a function
$\phi$ on $\mathcal{R}$ satisfying $c \omega \leq \ddb\phi \leq C\omega$ on $\mathcal{R}$
for a non-negative function $c$, and constant $C > 0$. More generally
$C$ could be a locally bounded function. For instance we could use
$\phi = r^2$ or $\phi = \log(1 + r^2)$. 

\begin{prop}\label{prop:L21}
  Suppose that $\alpha$ is a smooth $(0,1)$-form compactly supported in
  $\mathcal{R}$, such that $\dbar\alpha = 0$. Then there exists a
  function $f$ on $\mathcal{R}$ satisfying $\dbar f = \alpha$, and 
  \[ \int_{\mathcal{R}} |f|^2 e^{-\phi} \, \omega^n \leq \int_{\mathcal{R}}
    c^{-1} |\alpha|^2 e^{-\phi}\,\omega^n. \]
\end{prop}
\begin{proof}
  We follow the argument from Demailly~\cite[Theorem 5.1]{Dem} to
  prove H\"ormander's $L^2$-estimate~\cite{Ho}, the
  difference being that in our case we do not know that $\mathcal{R}$
  admits a complete K\"ahler metric, and so we need a more careful
  argument for approximating $L^2$ forms with smooth forms of compact
  support. 

  As in Demailly's proof, our goal is to prove the inequality
  \[ \label{eq:ineq1} \left| \int_{\mathcal{R}} \langle \alpha, v \rangle
      e^{-\phi}\,\omega^n \right|^2 \leq \left(\int_{\mathcal{R}} c^{-1}
      |\alpha|^2 e^{-\phi}\,\omega^n\right) \Vert \dbar^*_\phi
    v\Vert^2 \]
  for all smooth $(0,1)$-forms $v$ with compact support in
  $\mathcal{R}$. To define $\dbar^*_\phi v$ we are viewing $v$ as a
  $(0,1)$-form valued in the trivial bundle with metric $e^{-\phi}$. 
The existence of the required function $f$ then follows
  from the Riesz representation theorem. 

  Given such a smooth $(0,1)$-form $v$ compactly supported away from $S$, 
  we can decompose $v=v_1 + v_2$ under the $L^2$-orthogonal
  decomposition
  \[L^2 = \mathrm{ker}\,\dbar\oplus (\mathrm{ker}\,\dbar)^{\perp}.\]
  Note that $\dbar^*_\phi v_2=0$, and also $\dbar^*_\phi v=0$ near $S$, and
  therefore $\dbar^*_\phi v_1=0$ near $S$. Since also by definition $\dbar
  v_1=0$, it follows that $v_1$ is a harmonic $(0,1)$-form near
  $S$. If $u$ is any harmonic $(0,1)$-form valued in a line bundle
  $L$ with curvature $F_{j\bar k}$, in an orthonormal frame we have
\[ \begin{aligned} 0 &= \Delta_{\bar\partial} u =
  (\bar\partial^*\bar\partial + \bar\partial\bar\partial^*)u \\
&= -\nabla_j\nabla_{\bar j} u_{\bar k} +
\nabla_j \nabla_{\bar k} u_{\bar j} - \nabla_{\bar k}\nabla_j u_{\bar
  j} \\
&= -\nabla_j\nabla_{\bar j} u_{\bar k} + R^{\bar p}_{\,\, \bar j
  j\bar k} u_{\bar p} + F_{j\bar k} u_{\bar j},
\end{aligned}  \]
and so
\[ \nabla_j\nabla_{\bar j} u_{\bar k} = R^{\bar p}_{\,\, \bar j
  j\bar k} u_{\bar p} + F_{j\bar k} u_{\bar j}. \]
We also have
\[ \begin{aligned} \nabla_{\bar j}\nabla_j u_{\bar k} &= \nabla_j\nabla_{\bar j}u_{\bar
  k} - R^{\bar p}_{\,\, \bar k j\bar j} u_{\bar p} - F_{j\bar j}
u_{\bar k} \\
&= F_{j\bar k} u_{\bar j} - F_{j\bar j}u_{\bar k}.
\end{aligned} \]
It follows that
\[ \label{eq:Deltanorm}
\begin{aligned} \Delta |u|^2 = \nabla_j\nabla_{\bar j} |u|^2 &= (\nabla_j\nabla_{\bar j} u_{\bar
  k}) \overline{u_{\bar k}} + |\nabla^{1,0} u|^2 + |\nabla^{0,1} u|^2 +
u_{\bar k} \overline{\nabla_{\bar j}\nabla_j u_{\bar k}} \\
&= |\nabla u|^2 + R_{p\bar k}u_{\bar
  p}\overline{u_{\bar k}} + F_{j\bar k}u_{\bar j}\overline{u_{\bar k}}
+ \overline{F_{j\bar k}} u_{\bar k}\overline{u_{\bar j}} -
\overline{F_{j\bar j}} u_{\bar k}\overline{u_{\bar k}} \\
&= |\nabla u|^2  + R_{p\bar k} u_{\bar
  p}\overline{u_{\bar k}} + 2 F_{j\bar k} u_{\bar
  j}\overline{u_{\bar k}} - F_{j\bar j} |u|^2. 
\end{aligned} \]
In our setting the Ricci curvature vanishes, and $F_{j\bar k}
= \partial_j\partial_{\bar k} \phi$. 
In particular in a neighborhood of $S$ we have 
  \[ \Delta |v_1|^2 \geq |\nabla v_1|^2 - Cn |v_1|^2. \] 
  It follows that $\Delta |v_1| \geq -C|v_1|$ near $S$, for a locally
  bounded function $C$. Since $|v_1|\in L^2(\mathcal{R})$, we can apply
  Lemma~\ref{lem:subharmonic} below, to obtain that $|v_1|\in
  L^\infty_{loc}$ on a neighborhood of $S$ in $X$ (i.e. $|v_1|$ can
  not blow up as we approach the singular set). 

  Let us now define cutoff functions $\eta_R$ and $\chi_\epsilon$ as
  follows. The function $\eta_R$ equals 1 in $B_R(0)$, vanishes
  outside of $B_{R+1}(0)$, and satisfies $|\nabla\eta_R| < 2$. The function
   $\chi_\epsilon$ equals 1 outside of
  the $2\epsilon$-neighborhood of $S$, and vanishes in the
  $\epsilon$-neighborhood of $S$. In addition since $S$ has
  codimension 4 (see \cite{CCT}), we can arrange that on compact sets $K$ we have
  $\Vert \nabla\chi_\epsilon\Vert_{L^2(K)} \to 0$ as $\epsilon\to 0$. 
  Since $\chi_\epsilon\eta_R v_1$ has
  compact support in $\mathcal{R}$, by the Bochner-Kodaira inequality, 
 \[ \int_{\mathcal{R}} c |\chi_\epsilon \eta_R v_1|^2
    e^{-\phi}\,\omega^n \leq \Vert \dbar (\chi_\epsilon \eta_R
    v_1)\Vert^2 + \Vert \dbar^*_\phi (\chi_\epsilon\eta_R v_1)\Vert^2, \]
  and it follows using the Cauchy-Schwarz inequality that
  \[ \label{eq:CS} \left| \int_{\mathcal{R}} \langle \alpha, \chi_\epsilon \eta_R v_1 \rangle
      e^{-\phi}\,\omega^n \right|^2 &\leq \int_{\mathcal{R}}
    c^{-1}|\alpha|^2e^{-\phi}\,\omega^n \int_{\mathcal{R}} c |\chi_\epsilon \eta_R v_1|^2
    e^{-\phi}\,\omega^n \\
    &\leq \left(\int_{\mathcal{R}}
    c^{-1}|\alpha|^2e^{-\phi}\,\omega^n\right) (\Vert \dbar (\chi_\epsilon \eta_R
    v_1)\Vert^2 + \Vert \dbar^*_\phi (\chi_\epsilon\eta_R v_1)\Vert^2). \]

  Now recall that $v_1$ is locally bounded near $S$, and $\Vert
  \nabla\chi_\epsilon\Vert_{L^2}$ can be made arbitrarily small. It
  follows that for fixed $R$
  if we let $\epsilon \to 0$, we have $\chi_\epsilon\eta_R v_1\to
  \eta_R v_1$,  $\dbar (\chi_\epsilon\eta_R v_1) \to
  \dbar (\eta_R v_1)$ and $\dbar^*_\phi (\chi_\epsilon\eta_R v_1) \to
  \dbar^*_\phi (\eta_R v_1)$ in $L^2$. 
  In addition, using that $|\nabla \eta_R| < 2$ and that the
  supports of $\eta_R$ exhaust $\mathcal{R}$ as $R\to\infty$, we have that
  $\eta_R v_1\to v_1$,  $\dbar (\eta_R v_1) \to
  \dbar v_1$ and $\dbar^* (\eta_R v_1) \to
  \dbar^* v_1$ in $L^2$ as $R\to \infty$. It follows from
  \eqref{eq:CS} that
  \[ \left| \int_{\mathcal{R}} \langle \alpha, v_1 \rangle
      e^{-\phi}\,\omega^n \right|^2 
    &\leq \left(\int_{\mathcal{R}}
    c^{-1}|\alpha|^2e^{-\phi}\,\omega^n\right) (\Vert \dbar
    v_1\Vert^2 + \Vert \dbar^*_\phi  v_1\Vert^2). \]
  The required inequality \eqref{eq:ineq1} now follows since the
  assumption $\dbar\alpha =0$ implies that $\alpha$ is orthogonal to
  $v_2$, and at the same time $\dbar v_1=0$ and $\dbar^*_\phi
  v_1=\dbar^*_\phi v$. 
\end{proof}

We used the following Lemma in the proof above. Note that this
estimate fails when the singular set has codimension two. 
\begin{lem}\label{lem:subharmonic}
  Let $B$ be a unit ball in $X$, and suppose that $u\in L^2(B)$
  is such that on $B\setminus S$ the function $u$ is smooth, non-negative, and
  $\Delta u \geq -Au$ for a constant $A > 0$. Then we have
  \[ \sup_{\frac{1}{2}B \setminus S} u \leq C \Vert u\Vert_{L^2(B)} \]
  for a constant $C$ depending on $A$, the dimension, and asymptotic volume
  ratio of $X$.
\end{lem}
\begin{proof}
  The function $\tilde{u}(s,x)=e^{\sqrt{A}s}u(x)$ on the product
  $\mathbf{R}\times B$ satisfies $\Delta_{\mathbf{R}\times X}\tilde u
  = A e^{\sqrt{A}s} u + e^{\sqrt{A}s} \Delta_X u \geq 0$, and the $L^2$-norm of $\tilde{u}$
  on $[-1,1]\times B$ can be bounded in terms of the $L^2$-norm of $u$
  on $B$, and the constant $A$. Using this we can reduce to the case
  $A=0$. 

  Given $y\in \frac{1}{2}B\setminus S$, let $ H_t(x,y)$ be the heat
  kernel on $X$ satisfying $\partial_t  H_t = \Delta_x  H_t$ (see
  Ding~\cite{Ding} for details on the heat kernel on tangent cones). In
  addition let $\eta$ be a cutoff function such that $\eta=1$ on
  $\frac{2}{3}B$ and $\eta=0$ outside of $\frac{3}{4}B$. We have
 \[ \label{eq:ddtineq}\partial_t \int_B u(x)\eta(x)  H_t(x,y)\,dx &= \int_B u \eta
   \Delta H_t \\
   &= \int_B u \Delta(\eta H_t) - \int_B \Big[ u H_t \Delta
   \eta + 2u \nabla\eta \cdot \nabla H_t\Big] \\
   &\geq \int_B u \Delta(\eta H_t) - C\Vert u\Vert_{L^2}. \]
  Here we used that $\Delta\eta \in L^2$, and $\nabla H_t$ is
  bounded independently of $t$ on the support of $\nabla\eta$. 

  We claim that the first term is non-negative for any $t > 0$. For this let
  $\chi_\epsilon$ be cutoff functions such that $1-\chi_\epsilon$ is
  supported in the $\epsilon$-neighborhood of $S$, and in addition
  $\nabla\chi_\epsilon, \Delta\chi_\epsilon \in L^2$. Such a choice is
  possible because of Remark $1.15$ on page $6$ of \cite{JN}. For a fixed $t >
  0$ let us write $\psi = \eta H_t$. Then 
  \[  \int_B \chi_\epsilon u \Delta\psi = \int_B u
    \Delta(\chi_\epsilon \psi) - \int_B \Big[ u\psi
    \Delta\chi_\epsilon + 2u\nabla \psi\cdot \nabla\chi_\epsilon]. \]
  Since now $\chi_\epsilon\psi \geq 0$ is compactly supported in
  $B\setminus S$, the first term is non-negative, while at the same
  time  the last two terms tend to zero as $\epsilon \to 0$. It
  follows that
  \[ \int_B u \Delta(\eta H_t) \geq 0, \]
  and so \eqref{eq:ddtineq} implies
  \[ \int_B u \eta H_1 &\geq \lim_{t\to 0} \int_B u \eta H_t -
    C\Vert u\Vert_{L^2} \\
    &= u(y) - C\Vert u\Vert_{L^2}. \]
  Since $H_1$ is bounded above, we obtain the required estimate for
  $u(y)$. 
\end{proof}

\section{The $L^2$ estimate on tangent cones}\label{sec:L2est}
Suppose now that $(X,o)$ is a tangent cone at a point of a non-collapsed limit
space of $n$-dimensional 
K\"ahler manifolds with lower bounds on the Ricci curvature. In
particular, for all $R > 0$, the ball $B(o, R)\subset X$ is the
Gromov-Hausdorff limit of a sequence of balls $B(p_i, R)$ in K\"ahler
manifolds with $\mathrm{Ric} > -i^{-1}$, and $\mathrm{vol}(B(p_i, R))
> \nu R^{2n}$ for the non-collapsing constant $\nu > 0$, with $p_i\to
o$.  For sufficiently small $\epsilon$ we know from \cite{LiuSz} that
the $\epsilon$-regular set 
$\mathcal{R}_\epsilon$ is a complex manifold, and from
Cheeger-Jiang-Naber~\cite{CJN} that $X\setminus
\mathcal{R}_\epsilon$ is $(2n-2)$-rectifiable with locally finite
$(2n-2)$-dimensional Minkowski content. From now on we choose an
$\epsilon=\epsilon(n)$ sufficiently small, but fixed.

For $t\geq 0$, let $w(t)$ satisfy
\begin{equation}\label{w}
w' > 0, w''\leq 0, \text{ and } w'+tw''>0.
\end{equation}
For instance we can take $w(t)= C\log (t + 1)$ for any constant $C>0$.
In the $L^2$-estimate 
we will use a weight function of the form $\phi = w (r^2)$, where $r$
is the distance function from the vertex of the cone $X$. 
We denote by $\mu$ the natural measure on
$X$. It will also be useful to denote by $\mu_0$ a smooth volume measure
on $\mathcal{R}_\epsilon$ with respect to the holomorphic charts (up to bounded factors). It
follows that on compact sets we have a lower bound $d\mu > C^{-1}d\mu_0$.
\begin{prop}\label{prop:L2existence}
  Suppose that $\alpha = \bar\partial h$, where $h$
  is a smooth function compactly supported in
  $\mathcal{R}_\epsilon$. Then there exists a
  function $f$ on $\mathcal{R}_\epsilon$ satisfying $\dbar f = \alpha$, and 
  \[ \int_{\mathcal{R}_\epsilon} |f|^2 e^{-\phi} \, d\mu\leq
     \int_{\mathcal{R}_\epsilon} 
    \frac{|\alpha|^2}{w'(r^2)+r^2w''(r^2)} e^{-\phi}\, d\mu. \]
\end{prop}

\begin{rem}\label{rm1}
It will be clear from the proof that we can add weight functions with
logarithmic poles on $\mathcal{R}_\epsilon$. This will be useful to
separate tangents in Proposition~\ref{prop:separate}. 
\end{rem}
\begin{rem}
We must restrict ourselves to exact forms $\alpha$ for technical
reasons, but this will not matter in our application. We will define
the norm $|\alpha|$ precisely below by taking a limit of the corresponding norms
on smooth spaces. 
\end{rem}

As in the proof of Proposition~\ref{prop:L21}, we follow the approach in
Demailly~\cite{Dem} to H\"ormander's $L^2$ estimate, but we need to
take care with working on the non-smooth space
$\mathcal{R}_\epsilon$. We
construct $f$ satisfying the equation $\bar\partial f = \alpha$ in a
weak sense, i.e. satisfying
\[ \label{eq:weaksoln} \int_{\mathcal{R}_\epsilon} f\, \bar\partial\eta =
  \int_{\mathcal{R}_\epsilon} h\, \bar\partial\eta \]
for all smooth $(n,n-1)$-forms $\eta$ with compact support in
$\mathcal{R}_\epsilon$. In order to construct $f$ as an element of
$L^2(e^{-\phi}d\mu)$, we need to define the function $*\bar\partial \eta
\in L^2(e^{-\phi}d\mu)$. For this, as well as for the definition of
the norm $|\alpha|$, we observe that in any holomorphic chart on
$\mathcal{R}_\epsilon$ we have a well-defined tensor $g^{j\bar k}$
with bounded measurable components, corresponding to the inverse of the
metric. More precisely we have the following.
\begin{lem}\label{lem:inverseg}
Suppose that $q_i \in B(p_i, R)$ satisfy $q_i\to q \in
      \mathcal{R}_\epsilon$, and $z_{ij}$ are holomorphic charts on
      small balls $B(q_i, \rho)$ converging to a holomorphic chart
      $z_j$ on $B(q, \rho)$. We use the charts to identify functions
      on $B(q_i, \rho)$ with functions on $B(q,\rho)$. Then the
      inverses $g_i^{j\bar k}$ of the metric tensors on $B(q_i, \rho)$
      converge in $L^p(d\mu)$ for all $p$ to a tensor
      $g^{j\bar k}$ on $B(q,\rho)$ with bounded components. 
\end{lem}
\begin{proof}
  We show that we have
  \[ \int_{B(q, \rho)} |g_a^{j\bar k} - g_b^{j\bar k}|\,d\mu \to 0, \]
  as $a, b\to\infty$. Note first that we have a uniform upper bound
  $|g_a^{j\bar k}| < C$. Using that the regular points in
  $\mathcal{R}_\epsilon$ have full measure, together with
  a covering argument, it is enough
  to show that for all regular $x\in B(q,\rho)$ we have
  \[ \label{eq:a3} \lim_{\kappa\to 0} \lim_{a,b\to\infty} \fint_{B(x,\kappa)}
    |g_a^{j\bar k} - g_b^{j\bar k}|\,d\mu = 0. \]
  Consider a ball $B(x, \kappa)$, and $x_i\in B(p_i, R)$ such that
  $x_i\to x$. We can find holomorphic coordinates $w_{ip}$ on
  $B(x_i,\kappa)$ converging to $w_p$ on $B(x,\kappa)$, which give
  Gromov-Hausdorff approximations to the corresponding Euclidean
  balls. Let us rescale distances by $\kappa^{-1}$. The
  Cheeger-Colding estimate~\cite{CC2} shows that the corresponding
  metric components $g^{p\bar q}_i$ satisfy
  \[ \label{eq:a2} \fint_{\kappa^{-1}B(x_i, \kappa)} |g_i^{p\bar q} -
    \delta^{p\bar q}|\,d\mu_i < \Psi(\kappa, i^{-1} | x). \]
  Here, and below, $\Psi(\delta_1,\ldots,\delta_k|a_1,\ldots,a_l)$
  denotes a function converging to zero as $\delta_i\to 0$ while the
  $a_j$ are fixed. 
  In addition we have $g_{i,p\bar q} > (1-\Psi(\kappa, i^{-1}|x))
  \delta_{p\bar q}$. It follows from this and Colding's volume
  convergence~\cite{C} that in \eqref{eq:a2} we can 
  replace the measure $d\mu_i$ with $d\mu$. The gradient
  estimate for the coordinates $z_j$ implies that the Jacobian matrix
  $dz_j/dw_p$ is uniformly bounded. It follows that
  \[ \fint_{\kappa^{-1}B(x_i, \kappa)} |g_i^{j\bar k} -
    (dz_j/dw_p)\overline{(dz_k/dw_q)} \delta^{p\bar q}|\,d\mu <
    \Psi(\kappa, i^{-1} | x), \]
  which impiles \eqref{eq:a3}. The sequence $g_a^{j\bar k}$ therefore
  converges to a limit $g^{j\bar k}$ in $L^1(d\mu)$, and since these
  components are uniformly bounded, the convergence is in $L^p$ for
  all $p$ as well.

  If we have another set of charts $z_{ij}'$ converging to $z_j'$,
  then the corresponding components $g'^{j\bar k}$ are related to
  $g^{j\bar k}$ in the usual way. This follows from the fact that the
  transition functions $z_{ij}\circ z_{ij}'^{-1}$ converge to $z_j\circ
  z_j'^{-1}$ as $i\to \infty$. 
\end{proof}

In terms of the metric components $g^{j\bar k}$ we can now define
$|\alpha|^2 = g^{j\bar k} \alpha_{\bar k}\overline{\alpha_{\bar j}}$
in the usual way. Note that if we have local coordinates $z_{ij}$
converging to $z_j$ on $B(q,\rho)$ as above, such that $\phi_i,
\rho_i$ converge uniformly to $\phi, r^2$, then we have
\[  \label{eq:intalphaconverge}
  \lim_{i\to\infty} \int_{B(q_i, \rho)}
    \frac{|\alpha|_{g_i}^2 e^{-\phi_i}}{w'(\rho_i)+\rho_i w''(\rho_i)} \,d\mu_i = \int_{B(q, \rho)}
    \frac{|\alpha|^2 e^{-\phi}}{w'(r^2)+r^2w''(r^2)} \,d\mu. \]
This follows from the convergence of the metric components $g^{j\bar
  k}_i$ to $g^{j\bar k}$, together with the convergence~\cite{C} of the
measures $\mu_i$ to $\mu$. 

Similarly, we can define $*\bar\partial\eta$ in local coordinates by
letting
\[ \overline{*\bar\partial\eta} = e^\phi
  (\bar\partial\eta)_{1\bar 1\ldots n\bar n} \det(g^{j\bar k}). \]
In  terms of coordinates $z_{ij}$ converging to $z_j$ as above, we then
  have that $*_{\phi_i}\bar\partial\eta$ converges in $L^p$ to
  $*\bar\partial\eta$, and in addition we have 
      \[ \int_{\mathcal{R}_\epsilon} f \overline{*\bar\partial\eta}
        e^{-\phi}\,d\mu = \int_{\mathcal{R}_\epsilon} f
        \bar\partial\eta \]
      for any   $f\in L^2(e^{-\phi} d\mu)$. Note that such an $f$ is
      also locally in $L^2$ with respect to $d\mu_0$ since $\phi$ is
      locally bounded, and so the right hand side is well defined.

\begin{proof}[Proof of Proposition~\ref{prop:L2existence}]
Given a smooth $(n,n-1)$-form $\eta$ with compact support in
$\mathcal{R}_\epsilon$, our goal is to prove the inequality
\[ \label{eq:basicineq} \left(\int_{\mathcal{R}_\epsilon}
    h\,\overline{*\bar\partial\eta} e^{-\phi}\,d\mu\right)^2 &=
  \left(\int_{\mathcal{R}_\epsilon} h\,\dbar \eta\right)^2 \\ &\leq \int_{\mathcal{R}_\epsilon}
  \frac{|\alpha|^2 e^{-\phi}}{w'(r^2)+r^2w''(r^2)}\,d\mu \int_{\mathcal{R}_\epsilon}
  |*\dbar\eta |^2e^{-\phi}\,d\mu. \]
The existence of the required function $f$ then follows in the
standard way using the Hahn-Banach and Riesz representation
theorems. Indeed, if we denote by $E\subset L^2(e^{-\phi}d\mu)$ the closure of
the subspace of functions of the form $\overline{*\bar\partial\eta}$
with compact support, then \eqref{eq:basicineq} implies 
that $h$ defines a linear functional $\lambda: E\to\mathbf{C}$
with norm
\[ \Vert \lambda\Vert \leq \left(\int_{\mathcal{R}_\epsilon}
    \frac{|\alpha|^2 e^{-\phi}}{w'(r^2)+r^2w''(r^2)}\,d\mu\right)^{1/2}. \]
We then deduce the existence of $f\in E\subset L^2(e^{-\phi}d\mu)$ with the
same bound on its norm, such that
\eqref{eq:weaksoln} holds for all smooth $\eta$ with compact
support. It then follows that $\bar\partial f = \bar\partial h$. 

We prove the inequality \eqref{eq:basicineq} using approximations by
smooth spaces. 
Let $\mathcal{U}\subset \mathcal{R}_\epsilon$ be an open relatively
compact subset containing the closure of the support of $\eta$.
Let us fix a large radius $R > 0$, such that $\mathcal{U}\subset B(o,
R/2)$. Recall that $B(o, R)$ is a non-collapsed
Gromov-Hausdorff limit of a sequence of balls $B(p_i, R)$ in K\"ahler
manifolds with $\mathrm{Ric} > -i^{-1}$. For every
$q\in B(o, R)\cap \mathcal{U}$, by definition we have a radius $\delta$,
such that $B(q, \delta)$ is $\Psi(\epsilon)$-Gromov-Hausdorff close to the Euclidean
ball. Suppose that $q_i\in B(p_i, R)$ are such that $q_i\to q$. By
Theorem 2.1 in \cite{LiuSz}, we have $\delta'$ such that on each $B(q_i,
\delta')$ there is a holomorphic chart giving a Gromov-Hausdorff
approximation to the Euclidean ball, and a K\"ahler potential close to
$\frac{1}{2}d(q_i, \cdot)^2$. These holomorphic charts converge, as $i\to\infty$,
to a holomorphic chart on $B(q,\delta')$ defining the holomorphic
structure on $\mathcal{U}$. 

 We first construct suitable weight functions $\phi_i$ on $B(p_i,
  R)$. Since $B(p_i, R)\to B(o, R)$ in the Gromov-Hausdorff sense, by
  Cheeger-Colding~\cite{CC2} we have functions $b_i$ on $B(p_i, R)$
  such that $\Delta b_i^2/2 = n$, $|\nabla b_i| < 1 + \Psi(i^{-1})$ on
  $B(p_i, R-1)$, 
  \[ \label{eq:CC}\int_{B(p_i, R)} |\sqrt{-1} \partial\bar\partial b_i^2/2 - g_i|^2 + ||\nabla b_i|-1|^2 <
    \Psi(1/i), \]
  and $|b_i - r_i| < \Psi(1/i)$, in terms of the distance $r_i$
  from $p_i$. In other words $b_i^2/2$ is plurisubharmonic in an $L^2$
  sense, but this is not good enough to get an inequality of the form
  \eqref{eq:basicineq}. We will therefore modify $b_i^2/2$ on the set
  $\mathcal{U}$ to make it plurisubharmonic in a pointwise sense
  there. 

  Let us cover $\mathcal{U}\subset B(o, R/2)$ using charts $U_j$,
  such that the smaller charts $\frac{1}{10}U_j$ still cover. Under the
  Gromov-Hausdorff convergence these define charts $U_{ij}$ centered
  at points $q_{ij} \in B(p_i, R)$. 
On each $U_{ij}$ there is a K\"ahler
  potential $\psi_{ij}$, close to the function $\frac{1}{2} d(q_{ij},
  \cdot)^2$, the closeness determined by how close to the Euclidean
  ball our chart is in the Gromov-Hausdorff sense (this is controlled
  by our choice of $\epsilon(n)$ above). 

  At the same time, on each $U_{ij}$ we also have K\"ahler potentials
  $\phi_{ij}$, which are close to $\frac{1}{2}r_i^2$ (where $r_i$ is
  the distance from $p_i$ in the ball $B(p_i, R)$), in the sense
  that
  \[ |\phi_{ij} - r_i^2/2| < \Psi(1/i), \]
  where $r_i = d(p_i, \cdot)$ as above. This is because if we let
  $\psi_{\infty j}$ be the limit of the potentials $\psi_{ij}$ as $i\to\infty$ along
  a subsequence, then $\partial\bar\partial (\psi_{\infty j} - r^2/2) =
  0$ on $U_j$ (see \cite[Claim 3.1]{LiuSz} for a similar result). We can then define
  $\phi_{ij} = \psi_{ij} + (r^2/2 - \psi_{\infty j})$ using the local
  coordinates to identify $U_{ij}$ with $U_j$.

  It is clear that on the overlap $U_{ij}\cap U_{ik}$, $|\phi_{ij} -
  \phi_{ik}|<\Psi(1/i)$. We can ensure that on a smaller compact set
  of the intersection, $|\nabla(\phi_{ij} - \phi_{ik})|$ is small. 
  Using a partition of unity on $\bigcup_j \frac{1}{2}U_{ij}$, we can
  glue the
  $\phi_{ij}$ to obtain a function $\tilde{\rho_i}$ that
  satisfies 
  \[\label{eq0} \left|\tilde{\rho_i} -
    \frac{r_i^2}{2}\right| + |\partial\overline\partial\tilde{\rho_i} -
    g_i|<\Psi(1/i),\\|\nabla\tilde{\rho_i}|\leq
    r_i +\Psi(1/i), |\nabla (\tilde\rho_i -
    b_i^2/2)|<\Psi(1/i),\]
  on $\bigcup_j \frac{1}{3}U_{ij}$. 
  Note that since the partition of unity is defined using the charts,
  the functions giving the partition of unity have uniformly bounded
  derivatives.

  Similarly, we can define a cut-off function $\mu_i$, supported on
  $\bigcup_j\frac{1}{3}U_{ij}$, so that $\mu_i = 1$ on
  $\bigcup_j\frac{1}{6}U_{ij}$.  Furthermore, \begin{equation}\label{eq2}|\nabla\mu_i|+|\partial\overline\partial\mu_i|\leq C,\end{equation} where $C$ is independent of $i$.
 Now we define 
\begin{equation}\label{eq3}\rho_i = \mu_i\tilde{\rho_i} +
  (1-\mu_i)b_i+\epsilon_i,\end{equation} 
where $\epsilon_i$ is a positive sequence converging to zero so that $\rho_i\geq 0$.  Set \begin{equation}\label{eq4}\phi_i = w(\rho_i).\end{equation}
 
 From \eqref{eq0}, \eqref{eq2}, \eqref{eq3},\eqref{eq4} and standard computation, we find that on $\bigcup_j\frac{1}{6}U_{ij}$
\begin{equation}\label{eq5}
  \sqrt{-1}\partial\bar\partial \phi_i \geq (w'(\rho_i)+\rho_iw''(\rho_i)-\Psi(1/i))g_i, \end{equation}
and we also have \begin{equation}\label{eq6}
  \int_{B(p_i, R-1)} | \sqrt{-1}\partial\bar\partial \phi_i - (w'(\rho_i)g_i+\sqrt{-1}w''(\rho_i)\partial\rho_i\wedge\overline\partial\rho_i)|^2\leq \Psi(1/i).
  \end{equation}

  Let us consider now the $(n,n-1)$-form $\eta$. Let $\gamma_j$ be a
  partition of unity subordinate to the cover $U_j$ of
  $\mathcal{U}$. We can view each $\gamma_j \eta$ as an $(n,n-1)$ form
  on $U_{ij} \subset B(p_i, R)$, and define the form $\eta_i$ as their
  sum. Define $v_i = *_{\phi_i} \eta_i$, where we are viewing $\eta_i$ as
  an $(n,n-1)$-form with values in the trivial bundle $L^*$ with metric
  $e^{\phi_i}$ (i.e. the dual of the trivial bundle $L$ with metric
  $e^{-\phi_i}$), and $v_i$ as a $(0,1)$-form with values in $L$. On $B(p_i, R)$, under the
  $L^2$-product with weight $e^{-\phi_i}$ we decompose $v_i =
  v_i^{(1)} + v_i^{(2)}$, where $\bar\partial v_i^{(1)} = 0$, and
  $v_i^{(2)} \perp \mathrm{ker} \bar\partial$. It follows that
  $\bar\partial^*_{\phi_i} v_i^{(2)} = 0$, and so also
  $\bar\partial^*_{\phi_i} v_i^{(1)} = 0$ outside of the support of
  $v_i$. By Lemma~\ref{lem:Moser} below we have that \begin{equation}\label{eq7}|v_i^{(1)}|^2e^{-\phi_i} <
  C\Vert v_i\Vert_{L^2(e^{-\phi_i})}^2\end{equation} on the set $B(p_i, R-1)\setminus \bigcup
  \frac{1}{6} U_{ij}$, for a constant $C$ independent of $i$.  

  Let
  $\chi_R$ be a cutoff function, supported in $B(p_i, R-1)$, and equal
  to 1 on $B(p_i, R/2)$, such that $|\nabla\chi_R| < C'/R$, for a
  constant $C'$ independent of $R$. We can regard the $(0, 1)$-form
  $\chi_R v_i^{(1)}$ as a $(n, 1)$-form valued in the anti-canonical
  line bundle. From the Bochner-Kodaira formula (see $(4.7)$ of
  \cite{Dem}), 
  we have
  \[\label{eq:b1} \int_{B(p_i, R)} (\sqrt{-1}\partial_j\partial_{\bar k} &\phi_i + \mathrm{Ric}^{g_i}_{j\overline{k}})  (\chi_R v_i^{(1)})_{\bar j} \overline{ (\chi_R
      v_i^{(1)})_{\bar k}} e^{-\phi_i} \\ &\leq \Vert \bar\partial(\chi_R
    v_i^{(1)})\Vert^2_{L^2(e^{-\phi_i})}  + \Vert
    \bar\partial^*_{\phi_i}(\chi_R
    v_i^{(1)})\Vert^2_{L^2(e^{-\phi_i})}. \]

  To estimate the left hand side, note that writing $\mathcal{V}_i =
  \bigcup \frac{1}{6}U_{ij}$, \eqref{eq5} implies that we have
  \[ \int_{\mathcal{V}_i} (\sqrt{-1}\partial_j\partial_{\bar k} \phi_i
    + &Ric^{g_i}_{j\overline{k}}) (\chi_R
    v_i^{(1)})_{\bar j} \overline{ (\chi_R 
      v_i^{(1)})_{\bar k}} e^{-\phi_i} \\
      &\quad \geq 
      \int_{\mathcal{V}_i} (w'(\rho_i)+\rho_iw''(\rho_i)-\Psi(1/i)) |\chi_R v_i^{(1)}|^2 e^{-\phi_i}, \]
    while  also using \eqref{w}, \eqref{eq:CC}, \eqref{eq6} and \eqref{eq7} we have
  \[ &\int_{B(p_i, R)\setminus \mathcal{V}_i} (\sqrt{-1}\partial_j\partial_{\bar k} \phi_i + Ric^{g_i}_{j\overline{k}})(\chi_R
    v_i^{(1)})_{\bar j} \overline{ (\chi_R 
      v_i^{(1)})_{\bar k}} e^{-\phi_i} \\
    &\quad \geq  -\Psi(1/i)\int_{B(p_i, R)\setminus \mathcal{V}_i}|\chi_R v_i^{(1)}|^2
    e^{-\phi_i} + \int_{B(p_i, R)\setminus \mathcal{V}_i}  (w'(\rho_i)+\rho_iw''(\rho_i)) |\chi_R v_i^{(1)}|^2
  e^{-\phi_i} \\
  &\qquad - \int_{B(p_i,R)\setminus \mathcal{V}_i} |\sqrt{-1}
    \partial\bar\partial \phi_i - (w'(\rho_i)g_i+\sqrt{-1}w''(\rho_i)\partial\rho_i\wedge\overline\partial\rho_i)| |\chi_R v_i^{(1)}|^2
    e^{-\phi_i} \\
  &\qquad - \int_{B(p_i,R)\setminus \mathcal{V}_i}w''(\rho_i)(\rho_i(g_i)_{j\overline{k}} - \sqrt{-1}\partial_j\rho_i\wedge\partial_{\overline{k}}\rho_i) (\chi_R v_i^{(1)})_{\bar j} \overline{ (\chi_R
    v_i^{(1)})_{\bar k}} e^{-\phi_i}\\
  & \quad \geq \int_{B(p_i, R)\setminus \mathcal{V}_i}  (w'(\rho_i)+\rho_iw''(\rho_i))|\chi_R v_i^{(1)}|^2
    e^{-\phi_i} -\Psi(i^{-1}|R). \]

  Notice that we used the assumption $w''\leq 0$ above.
  As for the right hand side of \eqref{eq:b1}, we use that
  $\bar\partial v_i^{(1)} = 0$, and $\bar\partial^*_{\phi_i} v_i^{(1)}
  = \bar\partial^*_{\phi_i} v_i$, as well as the bound $|\nabla
  \chi_R| < C'/R$, to get
  \[  \Vert \bar\partial(\chi_R
    v_i^{(1)})\Vert^2_{L^2(e^{-\phi_i})} + \Vert
    \bar\partial^*_{\phi_i}(\chi_R
    v_i^{(1)})\Vert^2_{L^2(e^{-\phi_i})} \leq \Vert
    \bar\partial^*_{\phi_i} v_i\Vert^2_{L^2(e^{-\phi_i})} +
    \frac{C'}{R} \Vert v_i\Vert^2_{L^2(e^{-\phi_i})}. \]
  Note that by construction the $v_i$ are bounded independently of
  $i$, and so their $L^2$-norms are uniformly bounded. It follows that
  for sufficiently large $i$ we have
  \[\int_{B(p_i, R)} (w'(\rho_i)+\rho_iw''(\rho_i)) |\chi_R v_i^{(1)}|^2
    e^{-\phi_i} \leq \Vert
    \bar\partial^*_{\phi_i} v_i\Vert^2_{L^2(e^{-\phi_i})} +
    \frac{C'}{R}. \]

  We now use the assumption that $\alpha = \bar\partial h$ is
  exact. Using the cutoff functions $\gamma_j$ from before, we define
  smooth functions $h_i$ on $B(p_i, R)$ analogously to the way we
  defined $\eta_i$. We then let $\alpha_i = \bar\partial h_i$.
  By taking $R$ large, we can assume
  that the supports of $h_i, \alpha_i$ are in $B(p_i, R/2)$. 

  Since $\bar\partial^*_{\phi_i} v_i^{(2)} = 0$, we have
  \[ \label{eq:a1}
    &\left(\int_{B(p_i, R)} h_i \overline{\bar\partial^*_{\phi_i} v_i}
    e^{-\phi_i}\right)^2 = \left(\int_{B(p_i, R/2)} \langle \bar\partial h_i, \chi_Rv_i^{(1)}\rangle e^{-\phi_i}\right)^2 \\
  &\qquad\leq \int_{B(p_i, R/2)}
    \frac{|\alpha_i|^2e^{-\phi_i}}{w'(\rho_i)+\rho_iw''(\rho_i)} \int_{B(p_i R/2)} (w'(\rho_i)+\rho_iw''(\rho_i))
    |\chi_Rv_i^{(1)}|^2e^{-\phi_i} \\
  &\qquad\leq  \int_{B(p_i, R/2)}
    \frac{|\alpha_i|^2e^{-\phi_i}}{w'(\rho_i)+\rho_iw''(\rho_i)} \left( \Vert
      \bar\partial^*_{\phi_i} v_i\Vert^2_{L^2(e^{-\phi_i})} +
      \frac{C'}{R} \right),\]
  for $i$ large enough. Note that in the first equality we used
  that the support of $h_i$ is contained in $B(p_i, \frac{R}{2})$
  where $\chi_R = 1$. 

  Next we want to take the limit as $i\to \infty$. For this note first
  that
  \[ \int_{B(p_i,R)} h_i \overline{\bar\partial^*_{\phi_i} v_i}
    e^{-\phi_i} = \int_{B(p_i, R)} h_i\,
    \overline{*_{\phi_i}\bar\partial\eta_i} e^{-\phi_i}\,d\mu_i. \]
  Let us consider one of our charts $U_j$, with coordinates $z_{ik}$
  on $U_{ij}$ converging to $z_k$ on $U_j$. Note that in terms of
  these coordinates $\eta_i$ is not identified exactly with
  $\eta$ because of the way we defined $\eta_i$ in terms of the cutoff
  functions $\gamma_j$. However, if $U_j'$ is another chart, with corresponding
  coordinates $z_{ik}'$ converging to $z_k'$, then the transition
  functions $z_{ik}'\circ z_{ik}^{-1}$ converge smoothly to
  $z_k'\circ z_k^{-1}$. It follows that if we use our coordinates to
  identify $U_{ij}$ with $U_j$, then $\eta_i$ converges smoothly to
  $\eta$ as $i\to \infty$. Similarly $h_i$ converges smoothly to $h$ and
  in addition $\phi_i$ converges uniformly to
  $\phi$. It then follows from Lemma~\ref{lem:inverseg} that
  \[ \lim_{i\to\infty} 
\int_{B(p_i, R)} h_i\,
\overline{*_{\phi_i}\bar\partial\eta_i} e^{-\phi_i}\,d\mu_i =
\int_{B(p, R)} h\,
\overline{*\bar\partial\eta} e^{-\phi}\,d\mu.\]
  Similarly we can take the limit as $i\to\infty$ of the right hand
  side of \eqref{eq:a1}, noting that $\bar\partial^*_{\phi_i} v_i =
  *_{\phi_i}\bar\partial\eta_i$, and using Lemma~\ref{lem:inverseg}. 
  We therefore have
  \[ \left(\int_{B(p, R)} h\, 
\overline{*\bar\partial\eta} e^{-\phi}\,d\mu\right)^2 \leq  \int_{B(p, R/2)}
    \frac{|\alpha|^2e^{-\phi}}{w'(r^2)+r^2w''(r^2)}\,d\mu \left( \Vert *\bar\partial\eta\Vert^2_{L^2(e^{-\phi})} +
      \frac{C'}{R} \right),\]
  Finally, we obtain the required inequality \eqref{eq:basicineq} by letting $R\to
  \infty$. 
\end{proof}

We used the following estimate in the proof above. 
\begin{lem}\label{lem:Moser}
  Let $B(p,2)$ be a relatively compact ball in a K\"ahler manifold
  with $\mathrm{Ric} > -1$, and $\mathrm{vol}(B(p,2)) > \nu > 0$. 
  Let us denote by $L$ the trivial bundle with metric $e^{-\psi}$, and
  suppose that $|\nabla \psi|$ is bounded. Suppose that $u$ is a harmonic
  $(0,1)$-form valued in $L$. Then
  \[ \sup_{B(p,1)} |u| \leq C\left(\int_{B(p,2)}
      |u|^2\right)^{1/2}, \]
  where $C$ depends on $\nu$, the dimension and $\sup_{B(p,2)} |\nabla\psi|$. 
\end{lem}
\begin{proof}
  The proof is by Moser iteration. Note first that by
  \eqref{eq:Deltanorm}
 the norm of $u$ satisfies
  the differential inequality
  \[ \label{eq:Bochner1}
    \Delta |u|^2 &=  |\nabla u|^2 + R^{j\bar k}u_{\bar
      k}\overline{u_{\bar j}} + 2 g^{j\bar k} g^{p\bar q}
    (\nabla_j \nabla_{\bar q}\phi) u_{\bar k} \overline{u_{\bar
        p}} - g^{j\bar k} g^{p\bar q}
    (\nabla_j \nabla_{\bar k}\phi) u_{\bar q} \overline{u_{\bar
        p}} \\
    &\geq |\nabla u|^2 - |u|^2 + 2 g^{j\bar k} g^{p\bar q}
    (\nabla_j \nabla_{\bar q}\phi) u_{\bar k} \overline{u_{\bar
        p}} - g^{j\bar k} g^{p\bar q}
    (\nabla_j \nabla_{\bar k}\phi) u_{\bar q} \overline{u_{\bar
        p}}. \]
  Let us write $B_r = B(p,r)$ for any $ r\leq 2$. Let $\chi$ be a
  cutoff function supported in $B_2$, with $\chi=1$ on $B_{3/2}$, and
  $|\nabla\chi| + |\Delta\chi| < C$. Then we have (each
  integral being on $B_2$)
  \[ \label{eq:2} \int -\chi^2\Delta|u|^2 &\leq -\int \chi^2 |\nabla u|^2 + \int
    \chi^2 |u|^2 - 2\int \chi^2 g^{j\bar k} g^{p\bar q}
    (\nabla_j\nabla_{\bar q} \phi) u_{\bar k}\overline{u_{\bar p}} \\
    &\quad + \int \chi^2 g^{j\bar k} g^{p\bar q}
    (\nabla_j\nabla_{\bar k} \phi) u_{\bar q}\overline{u_{\bar p}} \\
    &= - \int\chi^2 |\nabla u|^2 + \int \chi^2 |u|^2 + 2\int 2\chi
    g^{j\bar k} g^{p\bar q} \nabla_j\chi\nabla_{\bar q}\phi\, u_{\bar
      k}\overline{u_{\bar p}} \\
    &\quad + 2\int \chi^2 g^{j\bar k} g^{p\bar q} \nabla_{\bar q}\phi
    \big( \overline{u_{\bar p}}\nabla_j u_{\bar k} + u_{\bar k}
    \overline{\nabla_{\bar j} u_{\bar p}}\big) \\
    &\quad - \int 2\chi
    g^{j\bar k} g^{p\bar q} \nabla_j\chi\nabla_{\bar k}\phi\, u_{\bar
      q}\overline{u_{\bar p}} \\
    &\quad - \int \chi^2 g^{j\bar k} g^{p\bar q} \nabla_{\bar k}\phi
    \big( \overline{u_{\bar p}}\nabla_j u_{\bar q} + u_{\bar q}
    \overline{\nabla_{\bar j} u_{\bar p}}\big) \\
    &\leq -\frac{1}{2} \int \chi^2|\nabla u|^2 + C\int |u|^2, \]
  for a constant $C$ depending on the dimension, $\nu$ and
  $\sup_{B(p,2)}|\nabla\phi|$. At the same time
  \[ \int \chi^2 \Delta |u|^2 = \int \Delta(\chi^2) |u|^2 \leq C \int
    |u|^2, \]
  and combining this with our previous inequality, we get
  \[ \label{eq:nablauC} \int \chi^2 |\nabla u|^2 < C \int |u|^2. \]

  Under our assumptions we have a bound for the Sobolev constant, and so
  \[ \left(\int |\chi |u||^{\frac{2n}{n-2}}\right)^{\frac{n-2}{n}}
    &\leq \int \left|\nabla(\chi|u|)\right|^2 \\
    &= \int |\nabla\chi|^2 |u|^2 + \int \chi^2 \big|\nabla |u|\big|^2
    + 2\chi |u| \nabla\chi.\nabla|u| \\
    &\leq C \int|u|^2,\]
  using the Cauchy-Schwarz inequality and also \eqref{eq:nablauC}.
  It follows that
  \[ \label{eq:step1}
     \Vert u\Vert_{L^{\frac{2n}{n-2}}(B_{3/2})} \leq C \Vert
     u\Vert_{L^2(B_2)}. \]

   To estimate higher $L^p$ norms, let $p \geq \frac{2n}{n-2}$, and
   let $\chi$ be a smooth function compactly supported in $B_2$. We
   can compute, using the Sobolev inequality
   \[ \label{eq:1} \left( \int \left| \chi|u|^{\frac{p}{2}}\right|^{\frac{2n}{n-2}}
     \right)^{\frac{n-2}{n}} &\leq C\int \left| \nabla (\chi
       |u|^{\frac{p}{2}})\right|^2 \\
     &= C\int |\nabla\chi|^2|u|^p + \chi^2 \left|\nabla
       |u|^{\frac{p}{2}}\right|^2 + 2\chi |u|^{\frac{p}{2}}
     \nabla\chi.\nabla |u|^{\frac{p}{2}}. \]
   We have
   \[ \int \chi^2 \left|\nabla
       |u|^{\frac{p}{2}}\right|^2 &= \int \frac{p^2}{8(p-2)} \chi^2 \nabla
     |u|^{p-2}. \nabla |u|^2 \\
   &= -\int \frac{p^2}{4(p-2)} \chi |u|^{p-2} \nabla\chi.\nabla |u|^2
   - \int \frac{p^2}{8(p-2)} \chi^2 |u|^{p-2} \Delta |u|^2, \]
 and
 \[ \int 2\chi |u|^{\frac{p}{2}}
     \nabla\chi.\nabla |u|^{\frac{p}{2}} &= \int \frac{p}{4} |u|^{p-2}
     \nabla \chi^2. \nabla |u|^2 \\
     &= -\int \frac{p(p-2)}{8} \chi^2|u|^{p-4} \nabla
     |u|^2.\nabla |u|^2 - \int \frac{p}{4} \chi^2 |u|^{p-2} \Delta
     |u|^2 \\
     &\leq  - \int \frac{p}{4} \chi^2 |u|^{p-2} \Delta
     |u|^2. \]
   
   Combining these with \eqref{eq:1} we get
   \[ \label{eq:3} \left( \int \left| \chi|u|^{\frac{p}{2}}\right|^{\frac{2n}{n-2}}
     \right)^{\frac{n-2}{n}} &\leq C\int |\nabla\chi|^2 |u|^p -
     \frac{p^2}{2(p-2)}\chi |u|^{p-1} \nabla\chi.\nabla |u|
     \\ &\qquad\qquad 
     -\left(\frac{p^2}{8(p-2)} + \frac{p}{4}\right) \chi^2
     |u|^{p-2}\Delta |u|^2 \]
   Using \eqref{eq:Bochner1} we have, similarly to \eqref{eq:2}
   \[ \label{eq:4} -\int \chi^2 |u|^{p-2} \Delta |u|^2 \leq - \frac{1}{2} \int \chi^2 |u|^{p-2}
     |\nabla u|^2 + Cp^2\int (\chi^2 + |\nabla\chi|^2) |u|^p, \]
   where $C$ depends on $\sup |\nabla\phi|$.
   Since $p\geq \frac{2n}{n-2}$, we have some constnat $c_n > 0$ such that
   \[ c_np < \frac{p^2}{2(p-2)} < c_n^{-1} p, \quad  c_np <
     \left(\frac{p^2}{8(p-2)} + \frac{p}{4}\right) < c_n^{-1}p. \]
   It then follows, combining \eqref{eq:3} and \eqref{eq:4}, and using
   the Cauchy-Schwarz inequality, that
   \[ \label{eq:5} \left( \int \left| \chi|u|^{\frac{p}{2}}\right|^{\frac{2n}{n-2}}
     \right)^{\frac{n-2}{n}} &\leq Cp^3 \int (\chi^2 + |\nabla\chi|^2)
     |u|^p. \]
   We now choose a sequence of cutoff functions $\chi_i$, such that
   $\chi_i$ is supported in $B_{1+2^{-i}}$, and $\chi_i=1$ on
   $B_{1+2^{-i-1}}$. We can arrange that $|\nabla\chi_i| < C
   2^{i}$. We also define $\gamma = \frac{n}{n-2}$. Equation
   \eqref{eq:5}, setting $\chi=\chi_i$ and $p=2\gamma^i$, implies that
   \[ \left(\int \chi_i^{2\gamma} |u|^{2\gamma^{i+1}}
       \right)^{\frac{1}{2\gamma^{i+1}}} \leq
       (8C\gamma^{3i} 2^{2i})^{\frac{1}{2\gamma^i}} \left( \int_{B_{1+2^{-i}}}
         |u|^{2\gamma^i}\right)^{\frac{1}{2\gamma^i}}. \]
     From this we have
     \[  \Vert u\Vert_{L^{2\gamma^{i+1}}(B_{1+2^{-i-1}})} \leq
       C^{\frac{3i}{\gamma^i}} \Vert
       u\Vert_{L^{2\gamma^i}(B_{1+2^{-i}})}. \]
     Iterating this, we find that
     \[ \sup_{B_1} |u| \leq C \Vert u\Vert_{L^{2\gamma}(B_{3/2})} \leq
       C' \Vert u\Vert_{L^2(B_1)}, \]
     using \eqref{eq:step1}. 
   \end{proof}

We will use Proposition~\ref{prop:L2existence} to construct
holomorphic functions on the tangent cone $X$. We will need the
following result, which implies that the functions that we construct
are actually harmonic across the singular set, and in particular
satisfy local $L^\infty$ and Lipschitz estimates. Note that in the
setting of Section~\ref{sec:Ricflat} with two-sided Ricci bounds, this
step is much more straightforward since the singular set has
codimension four.  
\begin{prop}\label{prop:harmonic}
  Let $f : X\to\mathbf{C}$ be a function such that $\dbar f=0$ on
  $\mathcal{R}_\epsilon$, and $f$ has polynomial growth in the sense
  that for some $k, D > 0$ we have
\[\label{eq:fpoly} \int_{B(o,R)} |f|^2 < DR^k, \]
  for all $R > 0$. Then $f$ is harmonic on $X$. In particular we have
  local estimates
  \[  \sup_{y\in B(x,1/2)} |f(y)| + \sup_{y,z\in B(x,1/2)}
    \frac{|f(y)-f(z)|}{d(y,z)} \leq C\left(\int_{B(x,1)}
      |f|^2\right)^{1/2}, \]
  for a constant $C$ depending only on the dimension and the
  non-collapsing constant $\nu$. In addition $f$ has polynomial growth
  in the pointwise sense that $|f(x)| \leq C(1 + d(o,x))^d$ for some
  $C, d > 0$ depending on $k,D$. 
\end{prop}
\begin{rem}
  In the statement we require the global growth condition
  \eqref{eq:fpoly} since we use the heat flow to smooth out $f$. The
  argument could be localized by working with Greens functions on
  balls instead, but we will not need this. 
\end{rem}
\begin{proof}
  Let $H(x,y,t)$ be the heat kernel on $X$. We let $f_t$ be the
  evolution of $f$ under the heat flow: 
  \[ f_t(x) = \int f(y) H(x,y,t) d\mu(y). \]
  Note that the polynomial growth assumption on $f$ together with the
  exponential decay of $H(x,y,t)$ ensures that this is well
  defined. Our goal is to show that $f_t=f$ for all $t >0$, which
  implies that $f$ is harmonic. We will use results about the
  convergence of heat kernels under Gromov-Hausdorff convergence due
  to Ding~\cite{Ding}, and argue by approximating $f_t$ with
  corresponding flows on smooth spaces. 

  We fix  $T\in(0,1)$, and $q\in\mathcal{R}_\epsilon\cap B(o,1)$. We will show
  that $f_T(q) = f_1(q)$. Let $R > 0$ be large, and set $\delta = e^{-R}$. The
  balls $B(p_i, R)$ converge to $B(o,R)\subset X$ in the
  Gromov-Hausdorff sense. By the results of Ding, the heat kernels
  $H_i(x,y,t)$ on $B(p_i,R)$ (with Dirichlet boundary conditions)
  satisfy
  \[ |H_i(x,y,t) - H(x,y,t)| < \Psi(i^{-1}, R^{-1}), \]
   for $t=T,1$ and $x,y\in B(o,R)$, where we use Gromov-Hausdorff approximations to
   identify points in $B(p_i,R)$ and $B(o, R)$. 
   
   Using that $X\setminus \mathcal{R}_\epsilon$ has locally finite
   codimension two Minkowki measure, we can find a cutoff function
   $\eta$ with compact support in $\mathcal{R}_\epsilon$, such that
   $\eta(q)=1$, and $\int_{B(o,R)} |\nabla\eta|^2 < \delta$. In
   addition we can ensure that $\int_{B(o,R)} (1-\eta)^2 < \delta$. We
   have corresponding cutoff functions $\eta_i$ on $B(p_i,R)$
   converging to $\eta$, and satisfying the same estimates. 

   Using the local holomorphic charts, we can find functions $f_i$ on $B(p_i,R)$,
   converging uniformly to $f$ on the support of $\eta$, such that
   $|\dbar f_i| < \Psi(i^{-1})$ on $\mathrm{supp}(\eta_i)$. In
   particular this implies that 
   \[ \int_{B(p_i, R)} |\dbar(\eta_i f_i)| < \Psi(i^{-1}, R^{-1}). \]
   Let $\chi_i$ denote cutoff functions supported in $B(p_i,R)$, and
   equal to 1 on $B(p_i, R-1)$, such that $|\nabla\chi_i| < C$ for a
   uniform constant $C$. Consider the heat flow for the functions
   $\chi_i\eta_i f_i$. We let
   \[ f_{i,R,t}(q) = \int_{B(p_i, R)} \chi_i(y)\eta_i(y) f_i(y)
     H_i(q,y,t)\,d\mu_i(y). \]
   We have
   \[ \frac{d}{dt} f_{i,R,t}(q) &= \int_{B(p_i,R)} \chi_i\eta_i f_i
     g_i^{a\bar b} \bar{\partial}_b \partial_a H_i(q,y,t)\, d\mu_i(y)
     \\
     &=  -\int_{B(p_i,R)} (\bar{\partial}_b\chi_i)\eta_i f_i
     g_i^{a\bar b} \partial_a H_i(q,y,t)\, d\mu_i(y) \\
&\quad - \int_{B(p_i,R)} \chi_i \bar{\partial}_b(\eta_i f_i)
     g_i^{a\bar b} \partial_a H_i(q,y,t)\, d\mu_i(y).
     \]
    On the support of $\nabla\chi_i$ we have the estimate $|\nabla_y
    H_i(q,y,t)| < C_T e^{-cR^2}$ for $t\in[T,1]$ for some $c > 0$ and
    $C_T > 0$ depending on $T$, while on $B(p_i,R)$ we have $|\nabla_y
    H(q,y,t)| < C_T$, for $t\in [T,1]$. It follows that
    \[ \left|\frac{d}{dt} f_{i,R,t}(q)\right| < C_T \Psi(i^{-1},
      R^{-1}), \]
    for $t\in [T,1]$ and so 
    \[ \label{eq:b-1}|f_{i,R,1}(q) - f_{i,R,T}(q)| < C_T \Psi(i^{-1}, R^{-1}). \]
    The convergence of $\eta_i f_i$ to $\eta f$, the convergence
    of the heat kernels as $i\to \infty$, and their exponential decay imply that
    \[ \label{eq:b2} |f_{i,R,t}(q) - f_t(q)| <  \Psi(i^{-1}|R), \]
    for $t=T, 1$. Choosing $R$ sufficiently large (depending
    on $T$) and $i$ large (depending on $T,R$), from \eqref{eq:b-1},
    \eqref{eq:b2} we see that $|f_1(q) - f_T(q)|$ can be made
    arbitrarily small, and therefore $f_1(q)=f_T(q)$.  
\end{proof}

\section{Tangent cones are affine varieties}\label{sec:affine}
As before, we suppose that $X$ is a tangent cone at a point of a
non-collapsed limit of K\"ahler manifolds with Ricci curvature bounded
below. Our main goal in this section is to prove that $X$ is
homeomorphic to an affine variety. Given the $L^2$ existence result
Proposition~\ref{prop:L2existence}, or Proposition~\ref{prop:L21} when
we have two-sided Ricci bounds, we can more or less follow the
argument in Donaldson-Sun~\cite{DS1,DS2} with suitable modifications as
in our previous work~\cite{LiuSz}, where only a lower bound
for the Ricci curvature is assumed. 

Recall that we have a subset $\mathcal{R}_\epsilon\subset X$ which has
the structure of a complex manifold, and whose complement has locally
finite codimension-two Minkowski content. Our first goal is to show that $X$ is
locally homeomorphic to a complex variety, by constructing holomorphic
functions that embed a neighborhood of the vertex $o\in X$ into
$\mathbf{C}^N$. The basic ingredient for this is the following. 

\begin{prop}\label{prop:separate}
 We can use holomorphic functions of polynomial growth on
  $X$ to separate points, and also to separate tangents at 
  points in $\mathcal{R}_\epsilon$. More precisely:
  \begin{enumerate}
\item Let $x_1, x_2\in X$. There exists a holomorphic function $f$ of
  polynomial growth on $X$ such that $f(x_1)\ne f(x_2)$. 
\item Let $x\in \mathcal{R}_\epsilon$. There are holomorphic functions
  $f_1,\ldots, f_n$ of polynomial growth on $X$ which define an
  embedding of a neighborhood of $x$ into $\mathbf{C}^n$. 
\end{enumerate}
\end{prop}
\begin{proof}
  The proof of this closely follows arguments in
  Donaldson-Sun~\cite{DS1} as well as the authors' work
  \cite{LiuSz}, given the $L^2$ existence result
  Proposition~\ref{prop:L2existence} together with
  Proposition~\ref{prop:harmonic}. We can construct holomorphic
  functions of polynomial growth by using a logarithmic weight
  $C\log(1+r^2)$, however in this case the curvature of the
  metric $e^{-\phi}$ is not proportional to the metric on the cone. We can however use a
  slightly different weight function $\phi = w(r^2)$, where $w(t)=t+1$ for $t$
  close to zero, and $w(t) = C\log(1+t)$ for $t$ large, still
  satisfying the inequalities \eqref{w}. Then in a neighborhood of the
  vertex the curvature of $e^{-\phi}$ will be the metric on $X$. Then we can proceed
  as in ~\cite{DS1} or \cite{LiuSz} to finish the proof of $(a)$ at
  least for $x_1,x_2$ close to the vertex. By scaling this can be
  extended to all of $X$. Part
  $(b)$ can be proven in a standard way using weight functions with logarithmic
  poles (see Remark \ref{rm1}).
  \end{proof} 

As in Chen-Donaldson-Sun~\cite[Section 2.5]{CDS2}, or
\cite[Proposition 3.2]{LiuSz}, we can
show that if at a point $x\in X$ there is a tangent cone that splits
off $\mathbf{R}^{2n-2}$, then polynomial growth holomorphic functions
can be used to embed a neighborhood of $x$ into $\mathbf{C}^n$. 

\begin{lem}
A small neighborhood of $o\in X$ is homeomorphic to a normal analytic space.
\end{lem}
\begin{proof}
The argument is very similar to \cite{DS2}. By using the separation of
points, we can find $N$ polynomial growth holomorphic functions $f_1,
.., f_N$ which all vanish at $o$ so that $\sum\limits_{j=1}^N
|f_j|^2\geq c>0 $ on $\partial B(o, 1)$. One can verify, by using
cutoff functions along the singular set,
 that the image of $(f_1, .., f_N)$ near the origin of $\mathbb{C}^N$
 defines a $d$-closed positive locally rectifiable current of type
 $(N-n, N-n)$. According to the main theorem in \cite{King} (or
 equivalently, Theorem $1.3(a')$ of \cite{HP}), 
 the image of $(f_1, .., f_N)$ defines a complex analytic variety of
 dimension $n$ near the origin in $\mathbb{C}^N$. By adding more
 holomorphic functions, we can ensure as in \cite[Propositions 2.3,
 2.4]{DS2}  that this map
 is a homeomorphism locally, and the image is a normal variety. 
\end{proof}
\begin{rem}
Note that unlike in \cite{DS1, DS2} or \cite{LiuSz}, the holomorphic
functions that we construct on $X$ do not necessarily arise as  limits
of  sequences of holomorphic functions on the approximating manifolds.  
\end{rem}

In order to show that $X$ is homeomorphic to a normal affine variety,
we can use the method in Donaldson-Sun~\cite{DS2}, which in turn is
based on an idea of Van Coevering~\cite{VC} in the setting of
K\"ahler cones with smooth links. The main point is to
decompose polynomial growth holomorphic functions into  sums of
homogeneous holomorphic functions under the homothetic vector field
$r\partial_r$.  
Recall that by \cite[Theorem 2]{Liu1} there is a one parameter group
$\sigma_t$ of isometries acting on the tangent cone $X$, preserving the
distance function from the vertex $o\in X$. Note that since it acts
by isometries, the action of $\sigma_t$ preserves
$\mathcal{R}_\epsilon$. We need the following. 
\begin{prop}
  The group $\sigma_t$ acts by biholomorphisms on
  $\mathcal{R}_\epsilon$. In addition the vector field $v$ on
  $\mathcal{R}_\epsilon$ generating this action satisfies $v = J
  r\partial_r$. 
\end{prop}
\begin{proof}
  Let us briefly recall the construction of the isometric action from
  \cite{Liu1}. We consider the geodesic annuli $A_i = B(p_i,
  10)\setminus B(p_i, 1/3)$. From Cheeger-Colding~\cite{CC2} 
  there are smooth functions $\rho_i$ on
  $A_i$ such that
  \[ \label{eq:CC1} \int_{A_i} |\nabla^2 \rho_i - g_i|^2 + |\nabla\rho_i - \nabla
    r_i^2/2|^2 &< \Psi(i^{-1}), \\
    \left| \rho_i - r_i^2/2\right| < \Psi(i^{-1}),  \quad   |\nabla\rho_i| \leq C, \]
  where $r_i$ is the distance function from $p_i$. Define vector
  fields $v_i = J\nabla\rho_i$, and let $\sigma_{i,t}$ be the
  diffeomorphisms generated by $v_i$. It is shown in \cite{Liu1} that
  for small $t$ we can extract a limit $\sigma_t$ as $i\to\infty$,
  which gives rise to a one parameter group of isometries on
  $B(o,6)\setminus B(o,5)$ in $X$. Since the action commutes with the
  homothetic transformations of $X$, it can be extended to all of $X$.  
  
  We will show that for small $t$ the limit $\sigma_t$ is a
  biholomorphism. Let $q\in \mathcal{R}_\epsilon$, such that $q\in
  B(o,6)\setminus B(o,5)$. Let $q_i\in B(p_i, 10)$ such that $q_i\to
  q$. We have a small $\delta > 0$ and holomorphic coordinates
  $z_{ij}$ on $B(q_i,\delta)$ converging to holomorphic coordinates
  $z_j$ on $B(q,\delta)$. 
  
  Abusing notation, we will also denote by $v_i$ the $(1,0)$-part
  of the real vector field defined above. 
  The estimate \eqref{eq:CC1} implies that 
  \[ \label{eq:b4} \int_{A_i} |\dbar v_i|^2 < \Psi(i^{-1}). \]
  Let us write $v_i = v_i^p \partial_{z_{ip}}$ on $B(q_i,\delta)$ in
  terms of the holomorphic coordinates. Since in these coordinates we
  have a lower bound $g_{i, a\bar b} > C^{-1} \delta_{a\bar b}$, it
  follows from \eqref{eq:b4} that the components $v_i^p$ satisfy
  \[ \int_{B(q_i, \delta)} |\dbar v_i^p|^2 < \Psi(i^{-1}). \]
  On the ball $B(q_i, \delta)$ we can apply the $L^2$-estimate to
  obtain holomorphic functions $\tilde{v}_i^p$ with 
  \[ \label{eq:vipL2} \int_{B(q_i, \delta/2)} |v_i^p - \tilde{v}_i^p|^2 <
    \Psi(i^{-1}). \]
  At the same time, using \eqref{eq:CC1}, we have uniform bounds
  $|v_i^p|, |\tilde{v}_i^p| < C$. Let us denote by
  $\tilde{\sigma}_{i,t}$ the flow of (the real part of) $\tilde{v}_i =
  \tilde{v}_i^p\partial_{z_{ip}}$. As long as $t < T$ for sufficiently
  small $T > 0$, this is well defined on $B(q_i, \delta/3)$, say. Up
  to choosing a subsequence we can assume that the $\tilde{v}_i$
  converge to a holomorphic vector field $\tilde{v}$ on
  $B(q,\delta/3)$, generating a one parameter group of biholomorphisms
  $\tilde{\sigma}_t$. 
  
  We claim that the diffeomorphisms $\sigma_{i,t}$ are approximated by
  $\tilde{\sigma}_{i,t}$ for $t < T$, and therefore the limiting flow
  $\sigma_t$ coincides with the flow $\tilde{\sigma}_t$ by
  biholomorphisms. More precisely, define the
  function
  \[ u(x,t) = |\sigma_{i,t}(x) - \tilde{\sigma}_{i,t}(x)|^2, \]
  for $x\in B(q_i, \delta/3)$ and $t < T$. We will show that we have
  $u(x,t) < \Psi(i^{-1})$ for $x$ outside of a set of measure at most
  $\Psi(i^{-1})$. Note that here we are
  viewing $B(q_i, \delta)$ as a subset of $\mathbf{C}^n$ using our coordinates. 
  Define
  \[ F(t) = \int_{B(q,\delta/3)} u(x,t). \]
  Note that
  \[ \frac{d}{dt} u(x,t) &= 2\big(\sigma_{i,t}(x) -
    \tilde{\sigma}_{i,t}(x)\big)\cdot \big(v_i(\sigma_{i,t}(x)) -
    \tilde{v}_i(\tilde{\sigma}_{i,t}(x)) \big) \\
&= 2 \big(\sigma_{i,t}(x) -
    \tilde{\sigma}_{i,t}(x) \big)\cdot \big(v_i(\sigma_{i,t}(x)) -
    \tilde{v}_i(\sigma_{i,t}(x))+ \tilde{v}_i(\sigma_{i,t}(x)) -
    \tilde{v}_i(\tilde{\sigma}_{i,t}(x)) \big). 
\]
Using that we have a gradient bound for the holomorphic vector field
$\tilde{v}_i$, this implies
\[ \frac{d}{dt} u(x,t) &\leq 2 \sqrt{u(x,t)} |v_i -
  \tilde{v}_i|(\sigma_{i,t}(x)) + C u(x,t) \\
   &\leq C u(x,t) + |v_i - \tilde{v}_i|_{Euc}^2 (\sigma_{i,t}(x)), \]
where we emphasize that the difference $|v_i-\tilde{v}_i|$ is measured
using the Euclidean metric given by our coordinates. 
We can now use the fact that $\sigma_{i,t}$ distorts volumes by at
most $\Psi(i^{-1})$ (see the proof of Theorem 2 in \cite{Liu1}),
together with the estimate \eqref{eq:vipL2} and the uniform bounds for
$v_i^p, \tilde{v}_i^p$ to see that
\[ \frac{d}{dt} F(t) \leq C F(t) + \Psi(i^{-1}). \]
Since $F(0) = 0$, we get $F(t) < \Psi(i^{-1})$ for $t < T$, and from
this it follows in turn that $u(x,t) < \Psi(i^{-1})$ outside of a set
of measure at most $\Psi(i^{-1})$, as required. 

Note that if in terms of the notation above we let $w_i = \nabla\rho_i$, and we let
$\phi_{i,t}$ denote the one-parameter group of diffeomorphisms
generated by $w_i$, then $\phi_{i,t}$ converges to the homothetic
expansion map on $X$, generated by $r\partial_r$ on
$\mathcal{R}_\epsilon$. 
Since $Jw_i=v_i$, it is not hard to see that in the limit we have $Jr\partial_r = v$. 
\end{proof}

The action of $\sigma_t$ extends to an isometric action of a torus $T$
on $X$, which also acts by biholomorphisms on
$\mathcal{R}_\epsilon$. Fixing $d > 0$, let $\mathcal{H}_d$ denote the
vector space of polynomial growth holomorphic functions $f$ on $X$ of
degree at most $d$, i.e. satisfying $|f(x)| < C(1 + d(o,x))^d$ for a
constant $C >0$. Since such $f$ are harmonic, $\mathcal{H}_d$ is
finite dimensional, and decomposes into weight spaces under the
$T$-action. It follows that any $f\in \mathcal{H}_d$ can be decomposed
into  a sum of eigenfunctions $f = f_{\alpha_1} + \ldots + f_{\alpha_m}$ for $\alpha_k\in
\mathrm{Lie}(T)^*$, where $e^{\sqrt{-1}t}\cdot f_{\alpha_k} = e^{\sqrt{-1}\langle
  \alpha_k,t\rangle}f_{\alpha_k}$ for $t\in\mathrm{Lie}(T)$. The
one-parameter group of isometries 
$\sigma_t$ is generated by a vector $\xi\in\mathrm{Lie}(T)$. The
relation  $v=Jr\partial_r$ then implies that $r\partial_r f_{\alpha_k}
= \langle \xi, \alpha_k\rangle f_{\alpha_k}$, i.e. each $f_{\alpha_k}$
is homogeneous. We can therefore embed a neighborhood of $o\in X$ into
$\mathbf{C}^N$ using homogeneous holomorphic functions, and this
embedding naturally extends to an embedding of all of $X$. As in
\cite[Lemma 2.19]{DS2} it follows that the image of $X$ is an affine
variety and by construction the torus $T$ acts linearly. This completes the proof of
Theorem~\ref{thm:main}.

The proof of Corollary~\ref{cor:1} follows exactly as  Proposition
2.21 and Lemma 4.2 of Donaldson-Sun~\cite{DS2}, based on the work of
Martelli-Sparks-Yau~\cite{MSY}.


\begin{thebibliography}{11}
\bibitem{And} Anderson, M. \emph{Convergence and rigidity of manifolds
    under Ricci curvature bounds}, Invent. Math. {\bf 97} (1990),
  429--445.
\bibitem{CC} Cheeger, J. and Colding, T. \emph{Lower bounds on Ricci
    curvature and the almost rigidity of warped products}, Ann. of
  Math (2), {\bf 144} (1996), 189--237.
\bibitem{CC2} Cheeger, J. and Colding, T. \emph{On the structure of spaces with Ricci curvature bounded below. I}, J. Differential Geom. {\bf 46}
(1997), no. 3, 406-480.
\bibitem{CCT}Cheeger, J., Colding, T. and Tian, G. \emph{On the
    singularities of spaces with bounded Ricci curvature},
  Geom. Funct. Anal. {\bf 12} (2002), 873-914.
\bibitem{CJN}Cheeger, J. and  Jiang, W. and Naber, A. \emph{Rectifiability of singular sets in spaces with Ricci curvature bounded below}, arXiv:1805.07988.
\bibitem{CDS2} Chen, X.-X., Donaldson, S., Sun, S. \emph{K{\"a}hler-{E}instein metrics on {F}ano manifolds. {II}:
  {L}imits with cone angle less than {$2\pi$}}, J. Amer. Math. Soc. \textbf{28}
  (2015), no.~1, 199--234.
\bibitem{C}Colding, T. {\em Ricci curvature and volume convergence},
  Ann. of Math. (2), {\bf 145} (1997), 477-501.
\bibitem{Dem} Demailly, J.-P. {\em Analytic methods in algebraic
    geometry}, Surveys of Modern Mathematics, 1. {\em International
    Press, Somerville, MA}, 2012. 
\bibitem{Ding} Ding, Y. {\em Heat kernels and Green's functions on
    limit spaces}, Comm. Anal. Geom. {\bf 10} (2002), no. 3, 475--514.
\bibitem{DS1} Donaldson, S.-K and Sun, S. {\em Gromov-Hausdorff limits of K\"ahler manifolds and algebraic geometry}, Acta Math. {\bf 213} (2014), no. 1, 63--106.
\bibitem{DS2} Donaldson, S.-K and Sun, S. {\em Gromov-Hausdorff limits of K\"ahler manifolds and algebraic geometry II}, J. Differential Geom. {\bf 107} (2017), no. 2, 327--371.
\bibitem{HP}Harvey, R and Polking, J. {\em Extending analytic
    objects}, Comm. Pure Appl. Math. {\bf 28} (1975), no. 6, 701--727.
\bibitem{Ho} H\"ormander, L. {\em $L^2$ estimates and existence
    theorems for the $\overline{\partial}$-operator}, Acta Math. {\bf
    113} (1965), 89--152.
\bibitem{JN} Jiang, W. and Naber, A. {\em $L^2$ curvature
    bounds on manifolds with bounded Ricci curvature}, arXiv:1605.05583.
\bibitem{King} King, J. {\em The currents defined by analytic
    varieties}, Acta. Math. {\bf 127} (1971), 185--220.
\bibitem{Liu1} Liu, G. {\em On the tangent cone of K\"ahler manifolds
    with Ricci curvature lower bound}, Math. Ann. {\bf 370} (2018),
  no. 1-2, 649--667.
\bibitem{LiuSz} Liu, G. and Sz\'ekelyhidi, G. {\em Gromov-Hausdorff
    limits of K\"ahler manifolds with Ricci curvature bounded below},
  arXiv:1804.08567
\bibitem{MSY} Martelli, D., Sparks, J. and Yau, S-T. {\em Sasaki-Einstein
    manifolds and volume minimisation}, Comm. Math. Phys. {\bf 280}
  (2007), 611--673. 
\bibitem{VC} Van Coevering, C. {\em Examples of asymptotically conical
    Ricci-flat K\"ahler manifolds}, Math. Z. {\bf 267} (2011),
  no. 1-2, 465--496
\end{thebibliography}
\end{document}